\documentclass[dvips,10pt]{article}
\title{\textsc{On finitely generated closures in\\ the theory of cutting planes}}
\author{Gennadiy Averkov\footnote{Institute for Mathematical Optimization, Faculty of Mathematics, University of Magdeburg, Universit\"atsplatz 2,  39106 Magdeburg, Germany; email: averkov@math.uni-magdeburg.de}}
\usepackage{amsthm,amsmath,amssymb,enumerate,graphicx,graphpap,empheq}
\usepackage{srcltx}

\newcommand{\rmcmd}[1]{\mathop{\mathrm{#1}}\nolimits}
\newcommand{\conv}{\rmcmd{conv}}

\newcommand{\bd}{\rmcmd{bd}}

\newcommand{\real}{\mathbb{R}}
\newcommand{\natur}{\mathbb{N}}
\newcommand{\integer}{\mathbb{Z}}
\newcommand{\rational}{\mathbb{Q}}

\newcommand{\term}[1]{\emph{#1}}

\newcommand{\cP}{\mathcal{P}}

\newcommand{\setcond}[2]{\left\{#1 \, : \, #2 \right\}}

\newcommand{\sprod}[2]{\left< #1 \, , \, #2 \right>}
\newcommand{\dotvar}{\, \cdot \,}
\newcommand{\intr}{\rmcmd{int}}

\newcommand{\vol}{\rmcmd{vol}}

\newcommand{\cX}{\mathcal{X}}
\newcommand{\vx}{\rmcmd{vert}}
\newcommand{\rec}{\rmcmd{rec}}

\newcommand{\cH}{\mathcal{H}}

\newcommand{\cL}{\mathcal{L}}
\newcommand{\overtwocond}[2]{\substack{{#1} \\ {#2}}}

\newtheorem{nn}{}[section]
\newtheorem{theorem}[nn]{Theorem}

\newtheorem{proposition}[nn]{Proposition}
\newtheorem{lemma}[nn]{Lemma}
\newtheorem{corollary}[nn]{Corollary}

\theoremstyle{definition}
\newtheorem{example}[nn]{Example}
\newtheorem{remark}[nn]{Remark}
\newtheorem*{acknowledgements*}{Acknowledgements}

\begin{document}

\maketitle

\begin{abstract}
	Let $P$ be a rational polyhedron in $\real^d$ and let $\cL$ be a class of $d$-dimensional maximal lattice-free rational polyhedra in $\real^d$. For $L \in \cL$ by $R_L(P)$ we  denote the convex hull of points belonging to $P$ but not to the interior of $L$.  Andersen, Louveaux and Weismantel showed that if the so-called max-facet-width of all $L \in \cL$ is bounded from above by a constant independent of $L$, then $\bigcap_{L\in \cL} R_L(P)$ is a rational polyhedron. We give a short proof of a generalization of this result. We also give a characterization for the boundedness of the max-facet-width on $\cL$. The presented results are motivated by applications in cutting-plane theory from mixed-integer optimization.
\end{abstract}

\newtheoremstyle{itsemicolon}{}{}{\mdseries\rmfamily}{}{\itshape}{:}{ }{}
\newtheoremstyle{itdot}{}{}{\mdseries\rmfamily}{}{\itshape}{:}{ }{}
\theoremstyle{itdot}
\newtheorem*{msc*}{2010 Mathematics Subject Classification} 

\begin{msc*}
	Primary 90C11;  Secondary 52A20, 52B20, 90C10.
\end{msc*}

% 52A20 (convex sets in $n$ dimensions)
% 52B20 (lattice polytopes)
% 90C11 (mixed integer programming)
% 90C10 (integer programming)

\newtheorem*{keywords*}{Keywords}

\begin{keywords*}
	Chv\'atal-Gomory closure; cutting plane; max-facet-width; mixed-integer optimization; split closure.
\end{keywords*}

\section{Introduction}

We use standard background from convex geometry; see, for example, \cite[Chapter~1]{MR1216521} and \cite[Part~III]{MR874114}. Let $d \in \natur$. By $o$ we denote the origin of $\real^d$. The standard scalar product of $\real^d$ is denoted by $\sprod{\dotvar}{\dotvar}$. For $n \in \natur$ we use the notation $[n]:=\{1,\ldots,n\}$.  Let $L$ be a $d$-dimensional polyhedron in $\real^d$. We introduce the functional $R_L$ by 
\begin{equation*} 
	R_L(X) := \conv(X \setminus \intr(L)),
\end{equation*}
where `$\conv$' and `$\intr$' stand for the convex hull and the interior, respectively, and $X \subseteq \real^d$. Assume that the polyhedron $L$ is rational. If $L \ne \real^d$ and the recession cone of $L$ is a linear space, then by $m(L)$ we denote the minimal value $m \in \natur$ such that $L$ can be given by
\begin{equation} 
	 L  =  \setcond{x \in \real^d}{b_i - m \le \sprod{a_i}{x} \le b_i \quad \forall i \in [n]}, \label{L:and:k:def}
\end{equation}
where
\begin{equation}
	n \in \natur, \  a_1,\ldots,a_n \in \integer^d \setminus \{o\}, \ b_1,\ldots,b_n \in \integer.	\label{integral:params}
\end{equation}
If $L=\real^d$ or the recession cone of $L$ is not a linear space, let $m(L):=+\infty$. With some further restrictions on $L$, the authors of \cite{MR2676765} use the term \term{max-facet-width} to refer to $m(L)$. It is not difficult to show that for $m(L)<+\infty$ the functional $R_L$ maps rational polyhedra to rational polyhedra. For a family $\cL$ of $d$-dimensional rational polyhedra in $\real^d$ we define 
\begin{equation} \label{m:for:family}
	m(\cL) := \sup_{L \in \cL} m(L).
\end{equation}
As an example to \eqref{m:for:family},  consider $\cL$ consisting of all \term{split sets} $L \subseteq \real^d$, i.e., sets of the form $L=\setcond{x \in \real^d}{i-1 \le \sprod{a}{x} \le i}$ with $a \in \integer^d \setminus \{o\}$ and $i \in \integer$. For such $\cL$ one has $m(\cL)=1$. In this note we present two theorems, which are motivated by \cite{MR2676765}. Our first theorem is a strengthening of the main result from \cite[Theorem~4.3]{MR2676765}.

\begin{theorem} \label{finite:generation:thm} 
	Let $P$ be a rational polyhedron in $\real^d$ and  let $\cL$ be a family of $d$-dimensional rational polyhedra in $\real^d$ satisfying $m(\cL) < +\infty$.
	Then there exists a finite subfamily $\cL'$ of $\cL$ such that every $L \in \cL$ satisfies  $R_{L'}(P) \subseteq R_L(P)$ for some $L'\in \cL'$.
\end{theorem}

Regarding Theorem~\ref{finite:generation:thm}, our  contribution is not so much the theorem itself as its short self-contained proof. Note that the complete proof of the corresponding Theorem~4.3 from \cite[\S\S2-4]{MR2676765} occupies nearly 18 pages. Our proof of Theorem~\ref{finite:generation:thm} employs basic facts from convex geometry and the well-known Gordan-Dickson lemma.

 In order to explain  the relation of Theorem~\ref{finite:generation:thm} to mixed-integer optimization we need several further notions. A subset $L$ of $\real^d$ is called \term{lattice-free} if $L$ is a $d$-dimensional closed convex set and $\intr(L) \cap \integer^d = \emptyset$. Furthermore, we call $L$ \term{maximal lattice-free} if $L$ is a lattice-free set which is not properly contained in another lattice-free set. Given a polyhedron $P$ in $\real^d$ and a family $\cL$ of $d$-dimensional polyhedra in $\real^d$, we call a closed halfspace $H$ an \term{$\cL$-cut} for $P$ if $H \supseteq P \setminus \intr(L)$ for some $L \in \cL$. If $\cL$ consists of lattice-free sets, one obviously has $P \cap \integer^d = P \cap H \cap \integer^d$ for every $\cL$-cut $H$ for $P$. The latter property is used by cutting-plane methods for solving integer and mixed-integer programs; for more details see \cite{MR2676765,MR0290793,conforti2011corner,DelPiaWeismantel10}. In particular we notice the well-known \term{intersection cuts}, which were introduced in \cite{MR0290793}, can be expressed in terms of $\cL$-cuts described above. The study of intersection cuts is an active area of research; see \cite{MR2676765,conforti2011corner,DelPiaWeismantel10} and the references therein for some of the recent contributions. We also refer to \cite{conf-et-al-polyhedral-approaches} for an overview on polyhedral approaches to mixed-integer optimization. 

As a consequence of Theorem~\ref{finite:generation:thm} we obtain.
\begin{corollary} \label{finite:generation:cor} Let $P$ and $\cL$ be as in Theorem~\ref{finite:generation:thm}. Then there exists a finite family $\cH$ of $\cL$-cuts such that each $H \in \cH$ is a rational halfspace and $\bigcap_{L \in \cL} R_L(P) =  \bigcap \cH.$ In particular, $\bigcap_{L \in \cL} R_L(P)$ is a rational polyhedron.
\end{corollary}

In the terminology of the cutting-plane theory the operation  $P \mapsto \bigcap_{L \in \cL} R_L(P)$ from Corollary~\ref{finite:generation:cor} is referred to as the \term{closure} (associated to the family of all $\cL$-cuts). Direct application of Corollary~\ref{finite:generation:cor} yields the results on polyhedrality of the \term{Chv\'atal-Gomory closure} and the \term{split closure} of a rational polyhedron (see  \cite{MR0313080,MR1059391,MR597387}). We also refer to two remarkable polyhedrality results of a somewhat different nature: the result from \cite{basu2011triangle} on polyhedrality of the so-called triangle closure and the result from \cite{MR2820903} on polyhedrality of the Chv\'atal-Gomory closure of an arbitrary compact convex set.

Since existing cutting-plane methods are based on lattice-free sets and since maximal lattice-free sets generate the strongest cuts within the family of all lattice-free sets, the study of families $\cL$ consisting of maximal lattice-free sets is of particular importance (see also \cite{ave-wag-2012,MR2855866,MR2832401} for related recent results). For such families $\cL$ in certaion situations the assumption of Theorem~\ref{finite:generation:thm} can be reformulated in an equivalent form. This is provided by our next theorem. Two sets $X, Y \subseteq \real^d$ are called \term{$\integer^d$-equivalent} if $Y = U(X)+b$, for some $d \times d$ unimodular matrix $U$ and a vector $b \in \integer^d$. A family $\cX$ of subsets of $\real^d$ is called \term{finite up to $\integer^d$-equivalence} if there exist finitely many sets $X_1,\ldots,X_t$ $(t \in \natur)$ in $\real^d$ such that each $X \in \cX$ is $\integer^d$-equivalent to some $X_i$ for $i \in [t]$.

\begin{theorem} \label{simple-reason-theorem}
	Let $\cL$ be a family of maximal lattice-free rational polyhedra in $\real^d$ such that $\dim (\conv (L \cap \integer^d)) = d$ for every $L \in \cL$. Then the following conditions are equivalent: 
	\begin{enumerate}[(i)]
		\item $m(\cL) < +\infty$; 
		\item $\cL$ is finite up to $\integer^d$-equivalence.
	\end{enumerate}
\end{theorem}

The implication (ii) $\Rightarrow$ (i) can be verified easily. Thus, (ii) is a `simple reason' of $m(\cL) < +\infty$. Theorem~\ref{simple-reason-theorem} asserts that, under the given assumptions, (ii) is the `only reason' of $m(\cL)<+\infty$. It might seem surprising that the assumption $\dim(\conv(L \cap \integer^d))=d$ in Theorem~\ref{simple-reason-theorem} is relevant for every $d \ge 3$. In fact, for every $d \ge 3$ one can construct a family $\cL$ of $d$-dimensional maximal lattice-free rational polyhedra with $\dim(\conv (L \cap \integer^d)) < d$ for every $L \in \cL$ and such that, for this family, Condition~(i) is fulfilled while Condition~(ii) is not fulfilled (see Example~\ref{counterexample} from Section~\ref{simple-reason-section}). In contrast to this, for $d=2$ the assumption $\dim(\conv (L \cap \integer^d))=d$ in Theorem~\ref{simple-reason-theorem} can be omitted. More precisely, as a consequence of Theorem~\ref{simple-reason-theorem} and a characterization of maximal lattice-free sets given by Lov\'asz in \cite[\S3]{MR1114315} we obtain the following.

\begin{corollary} \label{simple-reason-corollary}
	Let $\cL$ be a family of maximal lattice-free rational polyhedra in $\real^2$. Then the following conditions are equivalent: 
	\begin{enumerate}[(i)]
		\item $m(\cL) < +\infty$; 
		\item $\cL$ is finite up to $\integer^2$-equivalence.
	\end{enumerate}
\end{corollary}

\section{Proofs of Theorem~\ref{finite:generation:thm} and Corollary~\ref{finite:generation:cor}} \label{sect:proofs}

Given a polyhedron $P$ in $\real^d$, by $\vx(P)$ and $\rec(P)$ we denote the set of all vertices of $P$ and the recession cone of $P$, respectively. If $P$ is line-free, one has $P = \conv(\vx(P)) + \rec(P)$. The set $\gamma:= \setcond{x+ \lambda u}{ \lambda \ge 0}$ with $x \in \real^d$ and $u \in \real^d \setminus \{o\}$ is called the \term{ray} emanating from $x$ and having direction $u$. The well-known Gordan-Dickson lemma (see \cite{Gordan1899}) states that if $X$ is a subset of $\natur^d$ then there exists a finite subset $X'$ of $X$ such that every $x \in X$ satisfies $x' \le x$ for some $x' \in X'$ (here $\le$ is the standard partial order on $\real^d$ introduced by comparison of respective components). Apart from `$\conv$' and `$\intr$' we also use the abbreviations `$\bd$' and `$\vol$', which stand for the boundary and the volume, respectively. The following proposition has a straightforward proof.

\begin{proposition} \label{boundedness:prop}
	Let $L$ be a $d$-dimensional polyhedron in $\real^d$ such that $\rec(L)$ is a linear space. Let $\gamma$ be a ray in $\real^d$ with $\gamma \not\subseteq \intr(L)$. Then the set $\intr(L) \cap \gamma$ is bounded.
\end{proposition}
% \begin{proof}
% 	Choose $u \in \rec(\gamma) \setminus \{o\}$ and
% 	$x \in \gamma \setminus \intr(L)$. If $\intr(L) \cap \gamma$ was unbounded, then one would have $x+ \lambda u \in L$ for all sufficiently large $\lambda >0$ and $x + \lambda u \not\in L$ for all $\lambda<0$. The latter would yield $u \in \rec(L)$ and $-u \not\in \rec(L)$, contradicting the assumption on $\rec(L)$.
% \end{proof}

The following lemma seems to be folklore (see \cite[Lemmas~2.3 and 2.4]{MR2676765} and \cite[Corollary~11.3]{conf-et-al-polyhedral-approaches} for related statements). In order to keep the presentation self-contained, we give a short geometric proof.

\begin{lemma} \label{descr:R}
	Let $P$ be a line-free polyhedron in $\real^d$ and let $L$ be a $d$-dimensional polyhedron such that $\rec(L)$ is a linear space. Then $R_L(P)$ is a polyhedron. Furthermore, one has 
	\begin{equation} \label{equality:with:edges}
		R_L(P) = R_L(S) + \rec(P),
	\end{equation}
	where $S$ be the union of all edges of $P$.
\end{lemma}
\begin{proof}
	For proving \eqref{equality:with:edges} it suffices to verify the following inclusions:
	\begin{align}
		P \setminus \intr(L) & \subseteq R_L(S) + \rec(P), \label{incl:there} \\
		S \setminus \intr(L) + \rec(P) & \subseteq R_L(P). \label{incl:back}
	\end{align}
	Consider an arbitrary $x \in P \setminus \intr(L)$. By separation theorems, there exists a hyperplane $H$ containing $x$ and disjoint with $\intr(L)$. We have $x \in P \cap H = \conv(\vx(P \cap H)) + \rec(P \cap H)$, where $\vx(P \cap H) \subseteq S \cap H \subseteq S \setminus \intr(L)$ and $\rec(P \cap H) \subseteq \rec(P)$. This yields \eqref{incl:there}.  For showing \eqref{incl:back} we consider $x \in S \setminus \intr(L)$ and $u \in \rec(P)$ and derive $x+ u \in R_L(P)$. For $u=o$, one obviously has $x+u \in R_L(P)$. Let $u \ne o$. By Proposition~\ref{boundedness:prop}, the intersection of the ray $\gamma := \setcond{x + \lambda u }{ \lambda \ge 0}$ with $\intr(L)$ is bounded. Hence $x+u \in \gamma = R_L(\gamma) \subseteq R_L(P)$. This yields \eqref{incl:back}. 

	It remains to show that $R_L(P)$ is a polyhedron. Let $E$ be the set of all edges of $P$. If $e \in E$ and $e \setminus \intr(L)$ is bounded and nonempty, then $R_L(e) = \conv (\{u,v\})$ for some $u, v \in e \setminus \intr(L)$. If $e \in E$ and $e \setminus \intr(L)$ is unbounded, then $R_L(e) \subseteq w+ \rec(e)$ for some $w \in e \setminus \intr(L)$. Let $X$ be the finite subset of $S \setminus \intr(L)$ consisting of all $u, v$ and $w$ associated to edges $e \in E$ as above. We have 
	\begin{align*}
		R_L(S) + \rec(P) & = \conv\left(\bigcup_{e \in E} R_L(e)\right) + \rec(P) \subseteq \conv(X) + \rec(P) \subseteq R_L(S) + \rec(P).
	\end{align*}
	Thus, $\conv(X) + \rec(P) = R_L(S) + \rec(P) = R_L(P)$. Since $X$ is finite and $\rec(P)$ is a polyhedral cone, we deduce that $R_L(P)$ is a polyhedron.
\end{proof}

\begin{lemma} \label{integrality:lem}
	Let $L$ be a $d$-dimensional rational polyhedron in $\real^d$ such that $m=m(L) < +\infty$. Let $\gamma := \setcond{p+ \lambda u}{\lambda \ge 0}$ be a ray emanating from a point $p \in \intr(L)$ and having direction $u \in \integer^d \setminus \{o\}$. Assume that $\bd(L)$ and $\gamma$ intersect. Assume also that $p \in \rational^d$, and let $h \in \natur$ be such that $h p \in \integer^d$. Then the (unique) intersection point $q$ of $\bd(L)$ and $\gamma$ can be given by $q=p+\lambda u$, where $\lambda>0$ satisfies
	\[
		\frac{(h m)!}{\lambda} \in \natur.
	\]
\end{lemma}
\begin{proof}
	Let $L$ be given by \eqref{L:and:k:def}--\eqref{integral:params}. Since $q \in \bd(L)$, there exists $i \in [n]$ such that $\sprod{q}{a_i}=b_i-m$ or $\sprod{q}{a_i}=b_i$. We assume $\sprod{q}{a_i}=b_i$ (the case $\sprod{q}{a_i}=b_i-m$ being similar). Since $p \in \intr(L)$, one has $b_i-m < \sprod{a_i}{p} < b_i$. It follows that $\sprod{a_i}{u} \ne 0$ and we can express $\lambda$ by 
	\[
		\lambda = \frac{\sprod{a_i}{q} - \sprod{a_i}{p}}{\sprod{a_i}{u}} = \frac{b_i - \sprod{a_i}{p}}{\sprod{a_i}{u}} = \frac{h b_i - \sprod{a_i}{h p}}{h \sprod{a_i}{u}},
	\]
	where $h b_i - \sprod{a_i}{hp}$ is a natural number not larger than $h m$. It follows $\frac{ (h m)!}{\lambda} \in \natur$.
\end{proof}

\begin{proof}[Proof of Theorem~\ref{finite:generation:thm}] First we consider the case that $P$ is line-free.  We remark that if $\cL$ is represented as a finite union, say $\cL=\cL_1 \cup \cdots \cup \cL_t$ with $t \in \natur$, then it suffices to verify the assertion for each subfamily $\cL_i$ with $i \in [t]$ in place of $\cL$.  Let $E$ be the set of all edges of $P$. Given $L \in \cL$, we decompose $E$ into the following three sets:
\begin{align}
	E^+ &= \setcond{e \in E}{R_L(e) = e} & &\text{(the set of edges `preserved' by $L$),} \label{E+} \\
	E^- &= \setcond{e \in E}{R_L(e) =  \emptyset} & &\text{(the set of edges `removed' by $L$),} \label{E-} \\
	E^{\pm} &= \setcond{e \in E}{\emptyset \ne R_L(e) \ne e} & &\text{(the set of edges `bisected' by $L$).}\label{E+-}
\end{align}
In view of the remark on representation of $\cL$ as a finite union, without loss of generality we can assume that $E^+, E^-$ and $E^\pm$ do not depend on the choice of $L \in \cL$. That is, we assume that $E$ can be represented as a union of three sets $E^+, E^-$ and $E^\pm$ such that \eqref{E+}, \eqref{E-} and \eqref{E+-} hold for every $L \in \cL$. The degenerate case $E^\pm =\emptyset$ can be handled easily using Lemma~\ref{descr:R}. Assume $E^\pm \ne \emptyset$. Let $E^\pm = \{e_1,\ldots,e_s\}$, where $s \in \natur$. Let $i \in [s]$. By  Proposition~\ref{boundedness:prop}, the set $\intr(L) \cap e_i$ is bounded. Thus, taking into account that $e_i \in E^\pm$, we see that precisely one endpoint $p_i$ of $e_i$ belongs to $\intr(L)$. Let us choose any vector $u_i \in (\integer^d \cap \rec(e_i)) \setminus \{o\}$ if $e_i$ is unbounded and a vector $u_i \in \integer^d \setminus \{o\}$ such that $e_i \subseteq \conv (\{ p_i,p_i+u_i \})$ if $e_i$ is bounded. For every $L \in \cL$, the intersection point  of $e_i$ and $\bd(L)$ can be given by 
$p_i + \lambda_{i,L} u_i$,
where $\lambda_{i,L} > 0$. Let us fix $h \in \natur$ such that all vertices of $h P$ are integral points. By Lemma~\ref{integrality:lem}, for every $L \in \cL$ one has 
\[
	y_{L} := (h m)! \cdot  \left(\frac{1}{\lambda_{1,L}},\ldots,\frac{1}{\lambda_{s,L}} \right)^\top \in \natur^s.
\]
By the Gordan-Dickson lemma one can choose a finite subfamily $\cL'$ of $\cL$ such that for every $L \in \cL$ there exists $L'\in \cL'$ with $y_{L'} \le y_L$. For the  condition $y_{L'} \le y_L$ we have the following chain of equivalences:
\begin{align*}
	y_{L'} & \le y_L & &\Longleftrightarrow &  &\lambda_{i,L'} \ge \lambda_{i,L} & &\forall i \in [s] \\
	& & &\Longleftrightarrow & & R_{L'}(e)  \subseteq R_L(e) & & \forall e \in E^\pm \\
	& & &\Longleftrightarrow & & R_{L'}(P) \subseteq R_L(P). & &
\end{align*}
The last equivalence in the above chain follows from Lemma~\ref{descr:R}. Thus, $\cL'$ is a subfamily of $\cL$ with the desired properties.

Now, assume $P$ is not line-free, that is, the linear space $X:=\rec(P) \cap (-\rec(P))$ is strictly larger than $\{o\}$. If $L \in \cL$ and $X \not\subseteq \rec(L)$, one can choose a line $\gamma$ through $o$ which is contained in $X$ but not in $\rec(L)$. By the choice of $\gamma$, for every $x \in P$, the set $(x+\gamma) \cap \intr(L)$ is bounded. Hence $x \in x+\gamma = R_L(x+\gamma) \subseteq R_L(P)$. This shows the equality $P=R_L(P)$ for every $L \in \cL$ with $X \not\subseteq \rec(L)$. Thus, without loss of generality, we can assume $X \subseteq \rec(L)$ for every $L \in \cL$. The linear space $X$ is spanned by vectors from $\integer^d$. By this we can choose a basis $z_1,\ldots,z_k$ of the lattice $X \cap \integer^d$, where $k$ is the dimension of the linear space $X$. We extend $z_1,\ldots,z_k$ to a basis $z_1,\ldots,z_d$ of the lattice $\integer^d$. After a change of coordinates which transforms the basis $z_1,\ldots,z_d$ to the standard basis of $\real^d$ we can assume that $X= \real^k \times \{o'\}$, where $o'$ is the origin of $\real^{d-k}$. Then $P$ can be given by $P = \real^k \times P'$, where $P'$ is a line-free rational polyhedron in $\real^{d-k}$. Furthermore, every $L \in \cL$ can be given by $L =  \real^k \times L' $ for an appropriate rational polyhedron $L'$, which satisfies $m(L) = m(L')$. Taking into account the trivial equality $R_L(P) = \real^k \times R_{L'}(P')$, we see that, in the case that $P$ is not line-free, the assertion follows from the assertion for the case of line-free $P$.
\end{proof}

\begin{remark}
The main theorem from \cite[Theorem~4.3]{MR2676765} asserts that, for $P$ and $\cL$ as in Theorem~\ref{finite:generation:thm} and under the additional assumption that the elements of $\cL$ are maximal lattice-free sets, the set $\bigcap_{L \in \cL} R_L(P)$ is a rational polyhedron. Thus, in Theorem~\ref{finite:generation:thm} we both relax the assumptions and strengthen the assertion of the main result from \cite{MR2676765}. The motivation to relax the assumption is provided by the fact that Theorem~\ref{finite:generation:thm} can be used in the proof of the result on finite convergence of the so-called \emph{integral lattice-free closures} which was given in  \cite[Theorem~4]{DelPiaWeismantel10}. The authors of \cite[\S3]{DelPiaWeismantel10} indicate that they need to use a modification of the main result of \cite{MR2676765} with  weaker assumptions. On the other hand, the strengthened assertion gives a more detailed information on the family $\setcond{R_L(P)}{L \in \cL}$. 
\end{remark}

\begin{proof}[Proof of Corollary~\ref{finite:generation:cor}]
	The corollary is a straightforward consequence of Theorem~\ref{finite:generation:thm} and Lemma~\ref{descr:R}.
\end{proof}

\section{Proofs of Theorem~\ref{simple-reason-theorem} and Corollary~\ref{simple-reason-corollary}} \label{simple-reason-section}

We shall use the following description of maximal lattice-free sets given by Lov\'asz.

\begin{theorem} {\upshape \cite[\S3]{MR1114315}} \label{lovasz-char} Let $L$ be a lattice-free set in $\real^d$. Then the following statements hold.
\begin{enumerate}[I.]
	\item The set $L$ is maximal lattice-free if and only if $L$ is a polyhedron and the relative interior of each facet of $L$ contains a point of $\integer^d$.
	\item If $L$ is maximal lattice-free and unbounded, then $L$ is $\integer^d$-equivalent to $L'\times \real^k$, where $0 < k < d$ and $L'$ is a $(d-k)$-dimensional maximal lattice-free polytope in $\real^{d-k}$.
\end{enumerate}
\end{theorem}

Proofs of Theorem~\ref{lovasz-char} can be found in \cite[Theorem~1]{averkov2011proof} and \cite[Theorem~2.2]{MR2724071}.

\begin{lemma} \label{vol:versus:maxfw}
	Let $L$ be a $d$-dimensional rational polytope in $\real^d$. Then $\vol(L) \le m(L)^d.$
\end{lemma}
\begin{proof}
	Assume $m=m(L) < +\infty$, since otherwise the assertion is trivial. Let $L$ be given by \eqref{L:and:k:def}--\eqref{integral:params}. Since $L$ is bounded, there exist indices $1 \le j_1,\ldots,j_d \le n$ such that $a_{j_1},\ldots,a_{j_d}$ is a basis of $\real^d$. Let $A$ be the matrix with rows $a_{j_1}^\top,\ldots,a_{j_d}^\top$ (in this sequence) and let $b:=(b_{j_1},\ldots,b_{j_d})^\top$. We have
	\[
		L \subseteq \setcond{x \in \real^d}{b - A x \in [0,m]^d} = A^{-1}(b - [0,m]^d),
	\]
	where the matrix $A$ is integral. Hence $\vol(L) \le \frac{1}{|\det A|} m^d \le m^d$.
\end{proof}

\newcommand{\eps}{\varepsilon}

\begin{lemma} \label{circumscribed:finite:lem}
	Let $P$ be a $d$-dimensional integral polytope and let $m \in \natur$. Let $\cL$ be the family of all $d$-dimensional rational polytopes $L$ in $\real^d$ such that $L$ is a maximal lattice-free set, $\conv(L \cap \integer^d) = P$ and $m(L) = m$. Then $\cL$ is finite.
\end{lemma}
\begin{proof}
	We shall use the notions `distance' and `ball' in the standard Euclidean sense. By $\| \dotvar \|$ we denote the Euclidean norm of $\real^d$. Let $\delta > 0$ be the least possible distance between a pair of parallel hyperplanes $H^+$ and $H^-$ in $\real^d$ satisfying $P \subseteq \conv(H^+ \cup H^-)$. Let us choose $\rho>0$ such that $P$ is contained in the (closed) ball of radius $\rho$ with center at $o$. We consider an arbitrary $L \in \cL$ and assume that $L$ is given by \eqref{L:and:k:def}--\eqref{integral:params}. For $i \in [n]$ consider the hyperplanes
	\begin{align*}
		H^+_i &:= \setcond{x \in \real^d}{\sprod{a_i}{x} = b_i} & &\text{and} & H^-_i &:= \setcond{x \in \real^d}{\sprod{a_i}{x} = b_i-m}.
	\end{align*}
	If, for some $i \in [n]$, neither $L \cap H_i^+$ nor $L \cap H_i^-$ is a facet of $L$, then the corresponding inequalities $b_i-m \le \sprod{a_i}{x} \le b_i$ in \eqref{L:and:k:def} are redundant (that is, they can be dropped out without changing $L$). Thus, without loss of generality, we can assume that, for every $i \in [n]$, the set $L \cap H_i^+$ or $L \cap H_i^-$ is a facet of $L$. Taking into account this assumption and Theorem~\ref{lovasz-char}, we see that the intersection of $L \cap H_i^+$ or $L \cap H_i^-$ contains an integral point. Since all integral points of $L$ lie in $P$, we deduce that $L \cap H_i^+$ or $L \cap H_i^-$ contains a point of $P$.

	One has $P \subseteq \conv(H_i^+ \cup H_i^-)$, where $H_i^+$ and $H_i^-$ are parallel hyperplanes at distance $\frac{m}{\|a_i\|}$ from each other. From the definition of $\delta$ we deduce $\|a_i\| \le \frac{m}{\delta}$. The hyperplane $H_i^+$ is at distance $\frac{|b_i|}{\|a_i\|}$ from $o$. Analogously, the hyperplane $H_i^-$ is at distance $\frac{|b_i-m|}{\|a_i\|}$ from $o$. It follows that both $H_i^+$ and $H_i^-$ are at distance at least $ \frac{|b_i| - m}{\|a_i\|}$ from $o$. On the other hand, the hyperplane $H_i^+$ or $H_i^-$ contains a point of $P$. Hence, by the choice of $\rho$, the hyperplane $H_i^+$ or $H_i^-$ is at distance at most $\rho$ from $o$. We get $\frac{|b_i|-m}{\|a_i\|} \le \rho$, which implies $|b_i| \le \rho \|a_i\| + m \le \frac{m \rho }{\delta} + m$. Thus, for every $L \in \cL$ one can find a representation \eqref{L:and:k:def}--\eqref{integral:params} such that $\|a_i\| \le \frac{m}{ \delta}$ and $|b_i| \le \frac{m \rho}{\delta} + m$ for every $i \in [n]$. Since $\rho$ and $\delta$ depend only on $P$, we get the assertion.
\end{proof}

\begin{proof}[Proof of Theorem~\ref{simple-reason-theorem}]
	Assume (ii) is fulfilled. By Theorem~\ref{lovasz-char}, for every $L \in \cL$, $\rec(L)$ is a linear space of dimension at most $d-1$. Hence $m(L)<+\infty$ for every $L \in \cL$. Since the parameter $m(L)$ is invariant with respect to $\integer^d$-equivalence, (i) follows immediately. 

	Now, we assume (i) and show (ii). If $L \in \cL$ is unbounded then, by Theorem~\ref{lovasz-char}, $L$ is $\integer^d$-equivalent to $L'\times \real^k$ for some
	$0< k < d$ and a $(d-k)$-dimensional maximal lattice-free polytope in $L'$ in $\real^{d-k}$. Without loss of generality one can assume $L= L' \times \real^{d-k}$. Then $m(L)=m(L')$ and $\conv(L \cap \integer^d) = \conv(L'\cap \integer^k) \times \real^{d-k}$. In view of the latter relations, we see that it is sufficient to consider the case that $\cL$ consists of bounded polyhedra. By assumption, the family $\cP:= \setcond{\conv(L \cap \integer^d)}{L \in \cL}$ consists of $d$-dimensional integral polytopes.  In view of  Lemma~\ref{vol:versus:maxfw} the volume of each $P \in \cP$ is at most $m(\cL)^d$. Having an upper bound on the volume for the class of integral polytopes $\cP$ we deduce that $\cP$ is finite up to $\integer^d$-equivalence (this implication is well known; see, for example, \cite{MR1191566}). We choose finitely many integral polytopes $P_1,\ldots,P_t$ ($t \in\natur$) such that each $P \in \cP$ is $\integer^d$-equivalent to some $P_i$ for $i \in [t]$. Then (ii) follows by applying Lemma~\ref{circumscribed:finite:lem} with $P=P_i$ for each $i \in [t]$.
\end{proof}

\begin{proof}[Proof of Corollary~\ref{simple-reason-corollary}]
	In view of Theorem~\ref{lovasz-char}.II every unbounded element $L$ of $\cL$ is $\integer^2$-equivalent to $[0,1] \times \real$. By this without loss of generality we can assume that every $L \in \cL$ is bounded. Then $L$ has at leas three edges and, by Theorem~\ref{lovasz-char}, $\dim (\conv (L \cap \integer^2))=2$. Thus, the assumptions of Theorem~\ref{simple-reason-theorem} are fufillied and the assertion follows.
\end{proof}

\begin{example} \label{counterexample}
	As shown by Corollary~\ref{simple-reason-corollary}, the assumption $\dim(\conv(L \cap \integer^d))=d$ in Theorem~\ref{simple-reason-theorem} can be omitted for $d=2$. On the other hand, if the dimension $d \in \natur$ is at least $3$, then 
	the assumption $\dim(\conv(L \cap \integer^d)) =d$ cannot be omitted in general. In fact, for every $d \ge 3$ we shall construct a family $\cL$ of rational maximal lattice-free polyhedra in $\real^d$ satisfying $\dim (\conv(L \cap \integer^d)) < d$ for every $L \in \cL^d$ and such that $m(\cL)<+\infty$ but $\cL$ is \emph{not} finite up to $\integer^d$-equivalence. Thus, for $\cL$ as above Condition~(i) from Theorem~\ref{simple-reason-theorem} is fulfilled while Condition~(ii) is not. Below, whenever we consider a vector $x \in \real^d$ and $i \in [d]$ we denote by $x_i$ the $i$-th component of $x$. Our construction employs the cross-polytopes $C_d$ ($d \in \natur$) given by 
	\begin{align*}
		C_d  :=& \setcond{x \in \real^d}{|2 x_1 - 1| + \cdots + |2 x_d - 1| \le d} \\
			  =& \setcond{x \in \real^d}{a_1 (2 x_1 - 1) + \cdots + a_d (2 x_d -1) \le d \quad \forall a \in \{-1,1\}^d}.
	\end{align*}
	It can be shown using Theorem~\ref{lovasz-char}.I that $C_d$ is maximal lattice-free. Below we define $\cL$ in such a way that, for every $L \in \cL$, the intersection of $L$ with the horizontal coordinate hyperplane  $\real^{d-1} \times \{0\}$ coincides with $C_{d-1} \times \{0\}$ and the transformation $F \mapsto F \cap (\real^{d-1} \times \{0\})$ is a bijection from the set of facets of $L$ to the set of facets of its section $C_{d-1} \times \{0\}$.

	We shall need the following simple observation. If $A \in \integer^{d \times d}$ is a nonsingular matrix and $b \in \integer^d$, then for the nonsingular affine transformation $\phi :\real^d \rightarrow \real^d$ given by $\phi(x) = A x +b$ ($x \in \real^d$) and every $d$-dimensional rational polyhedron $P$ in $\real^d$ the inequality
	\begin{equation} \label{m:and:preimage:bound}
		m(\phi^{-1}(P)) \le m(P)
	\end{equation}
	holds. This follows directly from the definition of the max-facet-width (see \eqref{L:and:k:def} and \eqref{integral:params}). Consider the set $A_+^d$ (resp. $A_-^d$) of all vectors $a \in \{-1,1\}^d$ with even (resp. odd) number of entries equal to $-1$. Every vector $a \in \{-1,1\}^{d-1}$ (where $d \ge 2$) can be extended to a vector $(a_1,\ldots,a_{d-1}, t) \in A_+^d$, where $t \in \{-1,1\}$ is uniquely determined by $a$. The latter is also true for $A_-^d$ in place of $A_+^d$. Consequently, for every $d \ge 2$ one has 
	\begin{align}
		 |A_\pm^d| & = 2^{d-1} \label{card:power:two} \\
		 \sum_{a \in A_\pm^d} a & = o, \label{sum:to:zero}
	\end{align}
	where $A_\pm^d$ stands for $A_+^d$ or $A_-^d$. The family $\cL$ is introduced by applying certain nonsingular affine transformations to the rational polyhedron 
	\[
		P := \setcond{x \in \real^d}{ \sprod{a}{x} \le d \quad \forall a \in A_+^d}.
	\]
	We have $\dim (P) =d$ since $o \in \intr(P)$. The max-facet-width $m(P)$ of $P$ can be bounded from above using \eqref{card:power:two} and \eqref{sum:to:zero}. In fact, for every $a \in A_+^d$ and $x \in P$ one has 
	\[
		\sprod{a}{x} = - \sum_{b \in A_+^d \setminus \{a\}} \sprod{b}{x} \ge - \sum_{b \in A_+^d \setminus \{a\}} d = -(2^{d-1} - 1) d = d - d 2^{d-1},
	\]
	which yields
	\begin{equation}
		m(P) \le d 2^{d-1}.
	\end{equation}
	For $i \in [d]$ the transformation $a \mapsto (a_1,\ldots,a_{i-1},a_{i+1},\ldots,a_d)^\top$, which discards the $i$-th entry,  maps bijectively $\setcond{a \in A_+^d}{a_i=1}$ onto $A_+^{d-1}$ and $\setcond{a \in A_+^d}{a_i=-1}$ onto $A_-^{d-1}$. Using \eqref{card:power:two} and \eqref{sum:to:zero} the component $x_i$ of every point $x \in P$ (where $d \ge 3$ and $i \in [d]$) can be bounded from above and below as follows:
	\begin{align*}
		& 2^{d-2} d \ge \sum_{\overtwocond{a \in A_+^d}{a_i =1}} \sprod{a}{x} = 2^{d-2} x_i & &\text{and} & 
		& 2^{d-2} d \ge \sum_{\overtwocond{a \in A_+^d}{a_i =-1}} \sprod{a}{x} = - 2^{d-2} x_i. 
	\end{align*}
	We have shown $ |x_i| \le d$ for each $x \in P$ and $i \in [d]$, where $d \ge 3$. That is,
	\begin{equation}
		P \subseteq [-d,d]^d.
	\end{equation}
	In particular, $P$ is bounded and thus $\vol(P) < +\infty$. For every $k \in \natur$, consider the affine transformation $\phi_k : \real^d \rightarrow \real^d$ given by
	\begin{equation}
		\phi_k(x) := (2 x_1 - 1,\ldots,2 x_{d-1}-1, k x_d)^\top \quad \forall x \in \real^d.
	\end{equation}
  	Each $k \in \natur$ determines the rational polyhedron
	\[
		L_k:= \phi_k^{-1}(P)  = \setcond{x \in \real^d}{\sprod{a}{\phi_k(x)} \le d \quad \forall a \in A_+^d}.
	\]

	We introduce $\cL$ by $\cL := \setcond{L_k}{k \in \natur, \ k \ge 2 d}$. By construction $\vol(L_k) = \frac{1}{k 2^{d-1}} \vol(P)$. The volume of any two $\integer^d$-equivalent sets is the same. Hence, we deduce that $\cL$ is \emph{not} finite up to $\integer^d$-equivalence. By \eqref{m:and:preimage:bound}, $m(\cL)  = \sup_{k \ge 2 d} m(L_k) \le m(P) \le d 2^{d-1} < +\infty$. The bound $k \ge 2 d$ on $k$ guarantees that $L_k$ is rather thin in the vertical direction. More precisely, we have
	\[
		L_k = \phi_k^{-1} (P) \subseteq \phi_k^{-1}( [-d,d]^d ) \subseteq \real^{d-1} \times \left[-\frac{d}{k},\frac{d}{k}\right] \subseteq \real^{d-1} \times \left[-\frac{1}{2},\frac{1}{2}\right],
	\]
	which implies that all integer points of $L_k$ lie in the coordinate hyperplane $\real^{d-1} \times \{0\}$.
	By construction $L_k \cap (\real^{d-1} \times \{0\}) = C_{d-1} \times \{0\}$. Thus, $L_k \cap \integer^d = (C_{d-1} \times \{0\} ) \cap \integer^d = \{0,1\}^{d-1} \times \{0\}$. The latter yields $\dim(\conv(L_k \cap \integer^d)) = d-1 < d$. The relative interior of each facet of $L_k$ contains a point from $\{0,1\}^{d-1} \times \{0\}$. Consequently, by Theorem~\ref{lovasz-char}.I, every polytope from $\cL$ is maximal lattice-free.
\end{example}

\subsection*{Acknowledgements}

I would like to thank an anonymous referee for pointing to \cite{basu2011triangle}.

\small 

%\bibliographystyle{amsplain}
%\bibliography{generalized-closures}

\providecommand{\bysame}{\leavevmode\hbox to3em{\hrulefill}\thinspace}
\providecommand{\MR}{\relax\ifhmode\unskip\space\fi MR }
% \MRhref is called by the amsart/book/proc definition of \MR.
\providecommand{\MRhref}[2]{%
  \href{http://www.ams.org/mathscinet-getitem?mr=#1}{#2}
}
\providecommand{\href}[2]{#2}

\end{document}